\documentclass[11pt]{nyjm}
\usepackage{mathrsfs}
\usepackage{amsmath}
\usepackage{amsthm}
\usepackage{amsfonts}
\usepackage{bbold}
\usepackage{hyperref}
\usepackage{graphicx}
\usepackage{amsrefs}
\input xy
\xyoption{all}
\everymath{\displaystyle}

\title{Presentations and Tietze transformations of C*-algebras}

\author{Will Grilliette}
\address{Division of Mathematics\\
Alfred University\\
109B Myers Hall\\
Alfred, NY 14802\\}
\email{w.b.grilliette@gmail.com}



\newtheorem{thm}{Theorem}[section]

\newtheorem{cor}[thm]{Corollary}
\newtheorem{lem}[thm]{Lemma}

\theoremstyle{definition}
\newtheorem*{defn}{Definition}

\theoremstyle{remark}
\newtheorem{ex}[thm]{Example}

\DeclareMathOperator{\Ran}{ran}

\DeclareMathOperator{\CSAlg}{C^{*}Alg}
\DeclareMathOperator{\UCSAlg}{1C^{*}Alg}
\DeclareMathOperator{\unit}{Unit}

\newcommand{\alg}[1]{\mathcal{#1}}

\newcommand{\CSetC}{\mathbf{CSet}_{1}}
\newcommand{\CS}{\mathbf{C}^{*}}
\newcommand{\UCS}{\mathbf{1C}^{*}}
\newcommand{\Set}{\mathbf{Set}}

\begin{document}

\begin{abstract}
In this work, I develop a new view of presentation theory for C*-algebras, both unital and non-unital, heavily grounded in classical notions from algebra.  In particular, I introduce Tietze transformations for these presentations, which lead to a transformation theorem analogous to Tietze's 1908 result in group theory.
\end{abstract}

\maketitle
\tableofcontents

\section{Introduction}

Constructing universal C*-algebras by generators and relations is not a new idea, done previously in \cite{blackadar1985,gerbracht,loring1} among others.  This has classically been done by building a complex *-algebra, subject to certain *-algebraic relations, and then norming this structure according to a family of representations, subject to norm or positivity constraints.  However, this is counter to the usual algebraic means of building an object, which is to construct a free object and quotient it.

The primary philosophy of this paper is to reverse the order of the norming and quotienting processes, constructing a scaled-free C*-algebra like in \cite[\S3.2]{grilliette1} and quotienting by an ideal.  By proceeding in this way, more combinatorial and algebraic methods are unlocked, including the analogous calculus and transformation theorem to the classical work of \cite{tietze}.

The Tietze transformation theorem of \cite{tietze} is a classical result from combinatorial group theory.  In that work, any group can be represented by a presentation by generators and relations.  The transformational calculus allows the addition or removal of redundant generators and relations.  Moreover, two group presentations are isomorphic if and only if one presentation can be transformed into the other via this calculus.  Theorem \ref{tietze1} and Corollary \ref{tietze2} are the analogs for unital C*-algebras, and very similar proofs yield the same results for general C*-algebras.

This idea of C*-algebraic Tietze transformations was previously visited in \cite{gerbracht}, but only using *-algebraic conditions.  While the author of the current work does not know if a Tietze transformation theorem can be achieved with only *-algebraic conditions, he finds it unlikely.  In \cite{rordam1994}, it is known that the Cuntz algebra order 2, $\alg{O}_{2}$, satisfies $\alg{O}_{2}\otimes\alg{O}_{2}\cong\alg{O}_{2}$ as unital C*-algebras.  In \cite{ara2011}, the Leavitt path algebra order 2, $L_{2}$, is shown to be distinct from $L_{2}\otimes L_{2}$ as unital $k$-algebras.  Since the Cuntz algebra and Leavitt path algebra have identical *-algebraic presentations, this indicates that pure *-algebraic manipulation may be insufficient for C*-algebras.

Section \ref{construction} shows the construction of a C*-algebra presentation in this more algebraic sense for both the general and unital categories, summarizing major properties and features.  This relies on the category of normed sets with contractive functions from \cite[\S2.2]{grilliette1}.  Section \ref{transformations} develops the transformational calculus and theorem in analogy to \cite{tietze}.  Section \ref{examples} makes use of the transformations to compute examples, as well as the theorem to show failure of transformations to exist.

The author would like to thank the referees of this paper for their comments and patience in its revision.  He would also like to thank Prof.\ David Pitts for his advice and help in developing these ideas.  Thanks also go to Prof.\ Gene Abrams for the reference about the Leavitt path algebra.

\section{Building C*-algebras by generators and relations}\label{construction}

Here, a universal C*-algebra will be built by first constructing a scaled-free C*-algebra, analogous to the scaled-free Banach algebras of \cite[\S3.2]{grilliette1}.  Then, much like the classical algebraic model, a ``C*-relation'' will be an element of this structure, and the universal C*-algebra will be a quotient by an appropriate ideal.

\subsection{Scaled-free C*-algebras}

Recall the category of normed sets with contractive maps from \cite{grilliette1}, denoted $\CSetC$.  Likewise, let $\CS$ represent the category of C*-algebras with *-homomorphisms.  There is a natural forgetful functor $F_{\CS}^{\CSetC}:\CS\to\CSetC$ by dropping all of the algebraic properties, leaving only the norm functions and the contractivity of the maps.

For a normed set $(S,f)$, construction of a reflection along $F_{\CS}^{\CSetC}$ would proceed along natural lines.  One builds a free involutive $\mathbb{C}$-algebra on $S\setminus f^{-1}(0)$ and completes in an appropriate universal norm.  This is done previously in \cite[Satz 1.1.11]{gerbracht}, \cite[Proposition 2.2]{goodearl}, and \cite[Lemma 3.7]{loring2008}, so this work will only summarize the universal property of this object as in \cite{grilliette1}.  Let the C*-algebra built from $(S,f)$ in this manner be denoted by $\CSAlg(S,f)$ and $\theta_{S,f}:S\to\CSAlg(S,f)$ be the mapping of generators.

\begin{thm}[Reflection Characterization, $\CS$]\label{univprop2}
Given a C*-algebra $\alg{B}$ and a contractive map $\phi:(S,f)\to F_{\CS}^{\CSetC}\alg{B}$, there is a unique *-homomorphism $\hat{\phi}:\CSAlg(S,f)\to\alg{B}$ such that $F_{\CS}^{\CSetC}\hat{\phi}\circ\theta_{S,f}=\phi$.
\end{thm}

Further, since $(S,f)$ was arbitrary, the following functorial result is obtained.

\begin{cor}[Left Adjoint Functor, $\CS$]
There is a unique functor $\CSAlg:\CSetC\to\CS$ defined on objects as above, which is left adjoint to $F_{\CS}^{\CSetC}$.
\end{cor}

Similarly, one can build a scaled-free unital C*-algebra from $(S,f)$, though this can be done by appealing to the unitization.  To summarize, let $\UCS$ be the category of unital C*-algebras and unital *-homomorphisms and $F_{\UCS}^{\CS}:\UCS\to\CS$ the natural forgetful functor.  Given an arbitrary C*-algebra $\alg{A}$, the construction of the unitization is well-known in standard references such as \cite[Proposition I.1.3]{davidson}, where
\[
\unit(\alg{A}):=\alg{A}\oplus\mathbb{C}
\]
equipped with the appropriate multiplication, adjoint, and norm.  Letting $\iota_{\alg{A}}:\alg{A}\to\unit(\alg{A})$ be the inclusion map into the first summand in $\unit(\alg{A})$, the following is the universal property the unitization possesses in the notation of the current paper.

\begin{thm}[Reflection Characterization, Unitization]\label{unit-reflect}
Given a unital C*-algebra $\alg{B}$ and a *-homomorphism $\phi:\alg{A}\to F_{\UCS}^{\CS}\alg{B}$, there is a unique unital *-homomorphism $\hat{\phi}:\unit(\alg{A})\to\alg{B}$ such that $F_{\UCS}^{\CS}\hat{\phi}\circ\iota_{\alg{A}}=\phi$.
\end{thm}

That is, $\unit(\alg{A})$ is the smallest unitization of $\alg{A}$, provided $\alg{A}$ was not already unital.  Further, since $\alg{A}$ was arbitrary, the following functorial result is obtained.

\begin{cor}[Left Adjoint Functor, Unitization]
There is a unique functor $\unit:\CS\to\UCS$ defined on objects as above, which is left adjoint to $F_{\UCS}^{\CS}$.
\end{cor}

If $F_{\UCS}^{\CSetC}:\UCS\to\CSetC$ is the forgetful functor, a quick check shows that $F_{\UCS}^{\CSetC}=F_{\CS}^{\CSetC}\circ F_{\UCS}^{\CS}$, so one can invoke the closure of right adjoints on composition.  Hence, its left adjoint is given by $\UCSAlg:=\unit\circ\CSAlg$.  Letting $\eta_{S,f}:=\iota_{\CSAlg(S,f)}\circ\theta_{S,f}$ be the embedding of the generation set, the universal property can be stated as follows.

\begin{thm}[Reflection Characterization, $\UCS$]\label{univprop}
Given a unital C*-algebra $\alg{B}$ and a contractive map $\phi:(S,f)\to F_{\UCS}^{\CSetC}\alg{B}$, there is a unique unital *-homomorphism $\hat{\phi}:\UCSAlg(S,f)\to\alg{B}$ such that $F_{\UCS}^{\CSetC}\hat{\phi}\circ\eta_{S,f}=\phi$.
\end{thm}

Notice that $\CSAlg(S,f)$ is naturally identified as the codimension-1 ideal of $\UCSAlg(S,f)$ generated by $\eta_{S,f}(S)$.

\subsection{Presentations of C*-algebras}

Given any unital C*-algebra $\alg{A}$, notice that letting $(S,f):=F_{\UCS}^{\CSetC}\alg{A}$ and $\phi=id_{\alg{A}}$ in Theorem \ref{univprop} determines $\alg{A}$ as a quotient C*-algebra of $\UCSAlg(S,f)$.  Thus, the following definitions can be made, analogous to those found in pure algebra.

\begin{defn}
For a normed set $(S,f)$, a \emph{C*-relation} on $(S,f)$ is an element $r\in\UCSAlg(S,f)$.  A C*-relation $r$ on $(S,f)$ is \emph{non-unital} if $r\in\Ran\left(\iota_{\CSAlg(S,f)}\right)$.  Otherwise, $r$ is \emph{unital}.  An element of $\eta_{S,f}(S)$ itself is a \emph{generator}.
\end{defn}

\begin{defn}
For a normed set $(S,f)$ and C*-relations $R\subseteq\UCSAlg(S,f)$ on $(S,f)$, let $J_{R}$ be the two-sided, norm-closed ideal generated by $R$ in $\UCSAlg(S,f)$.  Then, the \emph{unital C*-algebra presented on $(S,f)$ subject to $R$} is
\[
\left\langle S,f|R\right\rangle_{\UCS}:=\UCSAlg(S,f)/J_{R},
\]
the quotient C*-algebra of $\UCSAlg(S,f)$ by $J_{R}$.
\end{defn}

The usual conventions for algebraic presentation theories are taken.  This includes blurring the distinction between $s\in S$ and $\left[\eta_{S,f}(s)\right]\in\langle S,f|R\rangle_{\UCS}$, writing relations equationally, or replacing relations with equivalent conditions.  For a finite set $S:=\left\{s_{1},\ldots,s_{n}\right\}$, then the normed set $(S,f)$ can be written as
\[
(S,f)=\left\{\left(s_{1},\lambda_{1}\right),\ldots,\left(s_{n},\lambda_{n}\right)\right\},
\]
where $\lambda_{j}:=f\left(s_{j}\right)$.  As such, the presentation of a unital C*-algebra can be written as
\[
\langle S,f|R\rangle_{\UCS}=:\left\langle\left(s_{1},\lambda_{1}\right),\ldots,\left(s_{n},\lambda_{n}\right)|R\right\rangle_{\UCS},
\]
directly associating a generator with its norm value from $f$.

This presentation notation is not far removed from the traditional notation of \cite{blackadar1985,loring2008}, except that the norm function $f$ from the original normed set is explicitly stated and kept.  This seems superficial, but it garners some pleasant consequences.

First, as $f$ imposes norm bounds on each of the generators, any formal presentation is guaranteed to yield a unital C*-algebra.  Likewise, any unital C*-algebra has a presentation of this form.  Philosophically, this emphasizes that the norm of a generator is part of the generator, rather than a condition to be imposed.

Second, a C*-relation as defined above includes unital *-polynomials in the elements of $S$ as described in \cite{blackadar1985,gerbracht}, but also elements created via the functional calculi.  This encodes many well-known functional analytic conditions as elements of a scaled-free unital C*-algebra, which can then be forced by quotienting.

\begin{ex}\label{relations}
\begin{enumerate}
Fix a unital C*-algebra $\alg{A}$ and $a\in\alg{A}$.
\item(Positivity and order) Let $p:\mathbb{R}\to\mathbb{R}$ by
\[
p(t):=\left\{\begin{array}{cc}
t,	&	t\geq 0,\\
0,	&	t<0.
\end{array}\right..
\]
Recall that $a\geq 0$ if and only if $a=p\left(\Re(a)\right)$, which is an application of the continuous functional calculus.
\item(Norm bounds on relations) For $\lambda\geq 0$, $\|a\|\leq\lambda$ if and only if $\left(a^{*}a\right)^{2}\leq\lambda^{2}a^{*}a$, which appeals to the above result.
\item(Invertibility) For $\lambda\geq 0$, there is $b\in\alg{A}$ satisfying $\|b\|\leq\lambda$ and $ba=\mathbb{1}$ if and only if $\mathbb{1}\leq\lambda^{2}a^{*}a$.  Moreover, $b\in C^{*}(\mathbb{1},a)$, appealing again to order.
\item(Analytic relations) The elements $\exp(a)$, $\sin(a)$, and $\cos(a)$ exist in any C*-algebra by the analytic functional calculus.
\end{enumerate}
\end{ex}

This fact has been explored previously in \cite[\S6]{hadwin2003} and \cite{loring2008} with different perspectives.  In \cite{hadwin2003}, a C*-algebraic relation corresponds to the zero-set of a non-commutative continuous function, while \cite{loring2008} considers a C*-algebraic relation as a full subcategory of a comma category between $\Set$ and $\CS$.

Lastly, while the notions of \cite{hadwin2003} and \cite{loring2008} are equivalent to the present work, the algebraic perspective taken here has the advantage of manipulating elements.  This will yield a Tietze calculus much like the classical result for group theory from \cite{tietze}.

As the presentation above is built from universal constructions, it satisfies a particular universal property, which follows directly from the constructions' universal properties.

\begin{thm}[Universal Property of a $\UCS$-Presentation]\label{univpresent}
Let $R$ be C*-relations on $(S,f)$ and $\alg{B}$ a unital C*-algebra.  Let $\phi:(S,f)\to F_{\UCS}^{\CSetC}\alg{B}$ be a contraction and $\hat{\phi}:\langle S,f|\emptyset\rangle_{\UCS}\to\alg{B}$ the unital *-homomorphism guaranteed by Theorem \ref{univprop}.  If $R\subseteq\ker\left(\hat{\phi}\right)$, then there is a unique unital *-homomorphism $\tilde{\phi}:\langle S,f|R\rangle_{\UCS}\to\alg{B}$ such that $\tilde{\phi}(s)=\phi(s)$.
\end{thm}

That is, if the C*-relations $R$ are satisfied by the image of $S$ under the contractive map $\phi$, there is a unique unital *-homomorphism from the presented unital C*-algebra to $\alg{B}$ extending $\phi$.

For general C*-algebras, an analogous presentation can be defined with similar conventions and properties.  Since $\CSAlg(S,f)$ embeds into $\UCSAlg(S,f)$, the term ``non-unital C*-relation'' will be used interchangeably for an element of $\CSAlg(S,f)$ or its image in $\UCSAlg(S,f)$.  The unitization also gives a nice formal connection between the two types of presentations.

To state this clearly, fix non-unital C*-relations $R$ on a normed set $(S,f)$.  Let $\alg{F}:=\langle S,f|\emptyset\rangle_{\CS}$ and $\iota_{\alg{F}}:\langle S,f|\emptyset\rangle_{\CS}\to\langle S,f|\emptyset\rangle_{\UCS}$ the unitization embedding.

\begin{thm}[Unitization of a $\CS$-Presentation]\label{unit-present}
Given a normed set $(S,f)$ and non-unital C*-relations $R$ on $(S,f)$,
\[
\unit\left(\langle S,f|R\rangle_{\CS}\right)\cong_{\UCS}\left\langle S,f\left|\iota_{\alg{F}}(R)\right.\right\rangle_{\UCS}.
\]
Furthermore, $\langle S,f|R\rangle_{\CS}$ is $\CS$-isomorphic to the ideal generated by $S$ in $\left\langle S,f\left|\iota_{\alg{F}}(R)\right.\right\rangle_{\UCS}$.
\end{thm}

The proof follows by mapping generators and using the universal properties of $\unit$ and the two presentations.

Lastly, if one restricts to *-algebraic relations, these $\CS$- and $\UCS$-presentations agree to those considered in \cite{blackadar1985,gerbracht,loring1} and others, where a *-algebra is created and then normed by a class of representations.  The proof is much like that of Theorem \ref{unit-present}, associating generators and using the universal properties.

\section{Formal manipulation of presentations}\label{transformations}

This section is dedicated to creating a formal calculus for presentations of C*-algebras and transformation theorem, analogous to the well-known group theoretic results of \cite{tietze}.

Tietze transformations for C*-algebras were considered previously in \cite[\S2.4.1]{gerbracht}, using only *-algebraic relations.  In the current work, Tietze transformations will be performed with C*-relations as defined in the previous section, which will not only yield a formal calculus, but also a transformation theorem.

The proofs of this section will be done for unital C*-algebra presentations, but the same results for general C*-algebra presentations hold by directly analogy.

\subsection{Tietze transformations for C*-algebras}

For this discussion, fix C*-relations $R$ on a normed set $(S,f)$, and define $\alg{F}:=\langle S,f|\emptyset\rangle_{\UCS}$.

The first Tietze transformation is the addition or removal of redundant C*-relations.  A set of C*-relations $Q\subseteq\alg{F}$ are \emph{redundant} for $R$ if $Q\subseteq J_{R}$, where $J_{R}$ is ideal generated by $R$ in $\alg{F}$.  That is, the conditions within $Q$ are implied by those within $R$.  Hence,
\[
\langle S,f|R\rangle_{\UCS}:=\alg{F}/J_{R}
=\alg{F}/J_{R\cup Q}
=:\langle S,f|R\cup Q\rangle_{\UCS},
\]
where $J_{R\cup Q}$ is the ideal generated by $R\cup Q$.

The second Tietze transformation is the addition or removal of redundant generators, i.e. one that can be written in terms of the others.  To clearly state this, let $G\subseteq\alg{F}$ and associate a new symbol $s_{g}$ and a nonnegative value $\lambda_{g}\in\left[\|g\|_{\alg{F}},\infty\right)$ to each $g\in G$.  Define $S_{G}:=\left\{s_{g}:g\in G\right\}$ and $f_{G}:S_{G}\to[0,\infty)$ by $f_{G}\left(s_{g}\right):=\lambda_{g}$, creating a new normed set $\left(S_{G},f_{G}\right)$.  Let
\[
\left(S_{1},f_{1}\right):=(S,f){\coprod}^{\CSetC}\left(S_{G},f_{G}\right),
\]
the disjoint union normed set and $\rho:(S,f)\to\left(S_{1},f_{1}\right)$ the canonical inclusion.  Applying $\UCSAlg$, define $\alg{F}_{1}:=\left\langle S_{1},f_{1}|\emptyset\right\rangle_{\UCS}$ and $\hat{\rho}:=\UCSAlg(\rho)$, which embeds $\alg{F}$ into $\alg{F}_{1}$ by association of generators.  Letting
\[
R_{1}:=\hat{\rho}(R)\cup\left\{s_{g}-\hat{\rho}(g):g\in G\right\},
\]
the mapping of generators from $\hat{\rho}$ can be lifted using Theorem \ref{univpresent} to give
\[
\langle S,f|R\rangle_{\UCS}\cong_{\UCS}\left\langle S_{1},f_{1}|R_{1}\right\rangle_{\UCS}.
\]

According to the philosophy of this paper, each new generator is added with a new corresponding norm bound.  This is necessary to avoid the behavior of the following example.

\begin{ex}
Notice that
\[
\left\langle\left. (x,1),\left(y,\frac{1}{4}\right)\right|x=x^{2},y=x^{*}x\right\rangle_{\UCS}\cong_{\UCS}\mathbb{C}
\]
and
\[
\left\langle(x,1)\left|x=x^{2}\right.\right\rangle_{\UCS}\cong_{\UCS}\mathbb{C}^{2}.
\]
\end{ex}

Thus, the norms of generators must be taken into consideration before removal.  Admittedly, the norms of the generators in this example are easily computed, but this may not so evident in general, particularly when applying the functional calculi.

\subsection{The free product}

The unital free product C*-algebra will be used briefly in the proof of the main theorem, and it will be useful in examples.  The following identification was previously considered for non-unital *-polynomial relations in \cite[Satz 3.3.2]{gerbracht}, and the current work extends the characterization to general C*-relations.

Within \cite[\S1.4]{freerandom}, the unital free product is shown to be the coproduct in $\UCS$, satisfying the appropriate mapping property.  As such, the notation ``${\coprod}^{\UCS}$'' will be used to denote the unital free product for an arbitrary index set.  When the index set is finite, the more common ``$*_{\mathbb{C}}$'' will be used.

To state this characterization clearly, fix an index set $\Gamma$.  For each $\gamma\in\Gamma$, fix C*-relations $R_{\gamma}$ on a normed set $\left(S_{\gamma},f_{\gamma}\right)$ and define $\alg{F}_{\gamma}:=\left\langle S_{\gamma},f_{\gamma}|\emptyset\right\rangle_{\UCS}$.  Let
\[
(S,f):={\coprod_{\gamma\in\Gamma}}^{\CSetC}\left(S_{\gamma},f_{\gamma}\right),
\]
the disjoint union normed set and $\rho_{\gamma}:\left(S_{\gamma},f_{\gamma}\right)\to(S,f)$ the canonical inclusion.  Applying $\UCSAlg$, define $\alg{F}:=\langle S,f|\emptyset\rangle_{\UCS}$ and $\hat{\rho}_{\gamma}:=\UCSAlg\left(\rho_{\gamma}\right)$ for each $\gamma\in\Gamma$, which embeds $\alg{F}_{\gamma}$ into $\alg{F}$ by association of generators.  Letting
\[
R:=\bigcup\hat{\rho}_{\gamma}\left(R_{\gamma}\right),
\]
the embeddings $\hat{\rho}_{\gamma}$ can be lifted using Theorem \ref{univpresent} into the connecting maps for the isomorphism
\[
\langle S,f|R\rangle_{\UCS}\cong_{\UCS}{\coprod_{\gamma\in\Gamma}}^{\UCS}\left\langle S_{\gamma},f_{\gamma}\left|R_{\gamma}\right.\right\rangle_{\UCS}.
\]
The proof of the isomorphism follows from gratuitous use of Theorem \ref{univpresent} and the coproduct realization of the unital free product.  For $\Gamma=\{1,2\}$, this states that
\[
\langle S,f|R\rangle_{\UCS}\cong_{\UCS}\left\langle S_{1},f_{1}\left|R_{1}\right.\right\rangle_{\UCS}*_{\mathbb{C}}\left\langle S_{2},f_{2}\left|R_{2}\right.\right\rangle_{\UCS}.
\]

\subsection{A Tietze transformation theorem for C*-algebras}

With an understanding of the different Tietze transformations, the main theorem can now be proven.  This proof is based on the treatment given in \cite[Section III.5]{baumslag} for group presentations.

For this discussion, only a pair of unital C*-algebras will be considered.  For $j=1,2$, fix a set of C*-relations $R_{j}$ on a normed set $\left(S_{j},f_{j}\right)$.  Define $\alg{F}_{j}:=\left\langle\left. S_{j},f_{j}\right|\emptyset\right\rangle_{\UCS}$, $\alg{A}_{j}:=\left\langle\left. S_{j},f_{j}\right|R_{j}\right\rangle_{\UCS}$, and $q_{j}:\alg{F}_{j}\to\alg{A}_{j}$ the quotient map.

To prove the theorem, one considers $\alg{A}_{1}$ and $\alg{A}_{2}$ as quotients of a single, unified algebra.  To build this structure, define
\[
(T,g):=\left(S_{1},f_{2}\right){\coprod}^{\CSetC}\left(S_{2},f_{2}\right)
\]
to be the disjoint union normed set, $\rho_{j}:\left(S_{j},f_{j}\right)\to(T,g)$ the canonical inclusions for $j=1,2$, and $\alg{G}:=\langle T,g|\emptyset\rangle_{\UCS}$.  Since $\UCSAlg$ is a left adjoint functor, it is cocontinuous, stating that
\[
\alg{G}\cong_{\UCS}\alg{F}_{1}*_{\mathbb{C}}\alg{F}_{2}
\]
with connecting maps $\hat{\rho}_{j}:=\UCSAlg\left(\rho_{j}\right)$ for $j=1,2$.

\[\xymatrix{
&	\alg{G}\\
\alg{F}_{1}\textrm{ }\ar@{->>}[d]_{q_{1}}\ar@{>->}[ur]^{\hat{\rho}_{1}}	&	&	\textrm{ }\alg{F}_{2}\ar@{->>}[d]^{q_{2}}\ar@{>->}[ul]_{\hat{\rho}_{2}}\\
\alg{A}_{1}	&	&	\alg{A}_{2}
}\]

The following lemma is the key step in the main result, allowing $\alg{A}_{1}$ and $\alg{A}_{2}$ to be realized as quotients of $\alg{G}$.  Further, the explicit C*-relations on $(T,g)$ are determined.

\begin{lem}\label{tietze0}
Given the notation above, assume $\Theta_{j}:\alg{G}\to\alg{F}_{j}$ is a unital *-homomorphism satisfying that $\Theta_{j}\circ\hat{\rho}_{j}=id_{\alg{F}_{j}}$.  Then, $\ker\left(q_{j}\circ\Theta_{j}\right)$ is the norm-closed, two-sided ideal $\alg{J}_{j}$ generated by
\[
\hat{\rho}_{j}\left(R_{j}\right)\cup\left\{s-\left(\hat{\rho}_{j}\circ\Theta_{j}\right)(s):s\in S_{3-j}\right\}
\]
in $\alg{G}$.
\end{lem}

\begin{proof}

For $r\in R_{j}$ and $s\in S_{3-j}$,
\[
\left(q_{j}\circ\Theta_{j}\right)\left(\hat{\rho}_{j}(r)\right)
=q_{j}\left(id_{\alg{F}_{j}}(r)\right)
=q_{j}(r)
=0
\]
and
\[\begin{array}{rcl}
\left(q_{j}\circ\Theta_{j}\right)\left(s-\left(\hat{\rho}_{j}\circ\Theta_{j}\right)(s)\right)	&	=	&	\left(q_{j}\circ\Theta_{j}\right)\left(s\right)-\left(q_{j}\circ\Theta_{j}\circ\hat{\rho}_{j}\circ\Theta_{j}\right)(s)\\
&	=	&	\left(q_{j}\circ\Theta_{j}\right)\left(s\right)-\left(q_{j}\circ id_{\alg{F}_{j}}\circ\Theta_{j}\right)(s)\\
&	=	&	\left(q_{j}\circ\Theta_{j}\right)\left(s\right)-\left(q_{j}\circ\Theta_{j}\right)(s)\\
&	=	&	0.\\
\end{array}\]
Hence, $\hat{\rho}_{j}\left(R_{j}\right)\cup\left\{s-\left(\hat{\rho}_{j}\circ\Theta_{j}\right)(s):s\in S_{3-j}\right\}\subseteq\ker\left(q_{j}\circ\Theta_{j}\right)$ so $\alg{J}_{j}\subseteq\ker\left(q_{j}\circ\Theta_{j}\right)$.

Let $\gamma:\alg{G}\to\alg{G}/\alg{J}_{j}$ be the quotient map.  For all $s\in S_{j}$ and $t\in S_{3-j}$,
\[\begin{array}{rcl}
\left(\gamma\circ\hat{\rho}_{j}\circ\Theta_{j}\right)(s)	&	=	&	\left(\gamma\circ\hat{\rho}_{j}\circ\Theta_{j}\circ\hat{\rho}_{j}\right)(s)\\
&	=	&	\left(\gamma\circ\hat{\rho}_{j}\circ id_{\alg{F}_{j}}\right)(s)\\
&	=	&	\left(\gamma\circ\hat{\rho}_{j}\right)(s)\\
&	=	&	\gamma(s)\\
\end{array}\]
and
\[\begin{array}{rcl}
\left(\gamma\circ\hat{\rho}_{j}\circ\Theta_{j}\right)(t)	&	=	&	\left(\gamma\circ\hat{\rho}_{j}\circ\Theta_{j}\right)(t)+\gamma\left(t-\left(\hat{\rho}_{j}\circ\Theta_{j}\right)(t)\right)\\
&	=	&	\left(\gamma\circ\hat{\rho}_{j}\circ\Theta_{j}\right)(t)+\gamma(t)-\left(\gamma\circ\hat{\rho}_{j}\circ\Theta_{j}\right)(t)\\
&	=	&	\gamma(t).\\
\end{array}\]
By Theorem \ref{univprop}, $\gamma=\gamma\circ\hat{\rho}_{j}\circ\Theta_{j}$.

For $b\in\ker\left(q_{j}\circ\Theta_{j}\right)$, then $\Theta_{j}(b)\in\ker\left(q_{j}\right)=J_{R_{j}}$, the norm-closed, two-sided ideal generated by $R_{j}$ in $\alg{F}_{j}$.  Thus, $\left(\hat{\rho}_{j}\circ\Theta_{j}\right)(b)\in\alg{J}_{j}$.  Also,
\[
\gamma\left(b-\left(\hat{\rho}_{j}\circ\Theta_{j}\right)(b)\right)
=\gamma(b)-\left(\gamma\circ\hat{\rho}_{j}\circ\Theta_{j}\right)(b)
=\gamma(b)-\gamma(b)
=0
\]
so $b-\left(\hat{\rho}_{j}\circ\Theta_{j}\right)(b)\in\ker(\gamma)=\alg{J}_{j}$.  Therefore, $b\in\alg{J}_{j}$.

\end{proof}

Now, the main result can be proven.

\begin{thm}[Tietze Theorem for $\UCS$]\label{tietze1}
$\alg{A}_{1}\cong_{\UCS}\alg{A}_{2}$ if and only if there is a sequence of four Tietze transformations changing the presentation of $\alg{A}_{1}$ into the presentation for $\alg{A}_{2}$.
\end{thm}

\begin{proof}

($\Leftarrow$) Since each Tietze transformation is an isomorphism, a sequence of four Tietze transformations determines an isomorphism between $\alg{A}_{1}$ and $\alg{A}_{2}$.

($\Rightarrow$) Assuming that $\alg{A}_{1}\cong_{\UCS}\alg{A}_{2}$, let $\phi:\alg{A}_{1}\to\alg{A}_{2}$ be a unital *-isomorphism.  First, maps $\Theta_{j}$ satisfying the conditions of Lemma \ref{tietze0} are created.  The purpose of these maps is to relate generators in $S_{1}$ in terms of generators in $S_{2}$, and vice versa.
\[\xymatrix{
&	\alg{G}\\
\alg{F}_{1}\textrm{ }\ar@{->>}[d]_{q_{1}}\ar@{>->}[ur]^{\hat{\rho}_{1}}	&	&	\textrm{ }\alg{F}_{2}\ar@{->>}[d]^{q_{2}}\ar@{>->}[ul]_{\hat{\rho}_{2}}\\
\alg{A}_{1}\ar[rr]_{\phi}^{\cong_{\UCS}}	&	&	\alg{A}_{2}
}\]
From \cite[Lemma 8.1.4]{loring1}, the unital C*-algebras $\alg{F}_{1}$ and $\alg{F}_{2}$ are projective with respect to all surjections in $\UCS$.  Thus, there is a unital *-homomorphism $\psi_{2}:\alg{F}_{2}\to\alg{F}_{1}$ such that $\phi\circ q_{1}\circ\psi_{2}=q_{2}$.  Using the coproduct characterization of $\alg{G}$, there is a unique unital *-homomorphism $\Theta_{1}:\alg{G}\to\alg{F}_{1}$ such that $\Theta_{1}\circ\hat{\rho}_{1}=id_{\alg{F}_{1}}$ and $\Theta_{1}\circ\hat{\rho}_{2}=\psi_{2}$.

Similarly, there is a unital *-homomorphism $\psi_{1}:\alg{F}_{1}\to\alg{F}_{2}$ such that $\phi^{-1}\circ q_{2}\circ\psi_{1}=q_{1}$.  Likewise, there is a unique unital *-homomorphism $\Theta_{2}:\alg{G}\to\alg{F}_{2}$ such that $\Theta_{2}\circ\hat{\rho}_{1}=\psi_{1}$ and $\Theta_{2}\circ\hat{\rho}_{2}=id_{\alg{F}_{2}}$.

Further, observe that
\[
\phi\circ q_{1}\circ\Theta_{1}\circ\hat{\rho}_{1}
=\phi\circ q_{1}\circ id_{\alg{F}_{1}}
=\phi\circ q_{1}
=q_{2}\circ\psi_{1}
=q_{2}\circ\Theta_{2}\circ\hat{\rho}_{1}
\]
and
\[
\phi\circ q_{1}\circ\Theta_{1}\circ\hat{\rho}_{2}
=\phi\circ q_{1}\circ\psi_{2}
=q_{2}
=q_{2}\circ id_{\alg{F}_{2}}
=q_{2}\circ\Theta_{2}\circ\hat{\rho}_{2}
\]
so by universal property of the coproduct, $\phi\circ q_{1}\circ\Theta_{1}=q_{2}\circ\Theta_{2}$.

Next, $\phi$ is to be decomposed into a composition of Tietze isomorphisms.  To this end, let $M_{j}:=\left\{s-\left(\hat{\rho}_{j}\circ\Theta_{j}\right)(s):s\in S_{3-j}\right\}\subset\alg{G}$ for $j=1,2$.  By Lemma \ref{tietze0}, $\ker\left(q\circ\Theta_{j}\right)$ is the norm-closed, two-sided ideal generated by $\hat{\rho}_{j}\left(R_{j}\right)\cup M_{j}$.  Observe that as $\phi$ is an isomorphism,
\[
\ker\left(q_{2}\circ\Theta_{2}\right)
=\ker\left(\phi\circ q_{1}\circ\Theta_{1}\right)
=\ker\left(q_{1}\circ\Theta_{1}\right).
\]
Thus, the ideal generated by $\hat{\rho}_{1}\left(R_{1}\right)\cup M_{1}$ is the same as the ideal generated by $\hat{\rho}_{2}\left(R_{2}\right)\cup M_{2}$.

Therefore, there are C*-relation-adding and generator-adding Tietze isomorphisms $\alpha,\beta,\sigma,\tau$ below.
\[\xymatrix{
\alg{A}_{1}\ar[r]^{\phi}_{\cong_{\UCS}}\ar[d]_{\alpha}^{\cong_{\UCS}}	&	\alg{A}_{2}\ar[d]^{\beta}_{\cong_{\UCS}}\\
\left\langle T,g\left|\hat{\rho}_{1}\left(R_{1}\right)\cup M_{1}\right.\right\rangle_{\UCS}\ar[d]^(.4){\sigma}_(.4){\cong_{\UCS}}	&	\left\langle T,g\left|\hat{\rho}_{2}\left(R_{2}\right)\cup M_{2}\right.\right\rangle_{\UCS}\ar[dl]_(.4){\tau}^(.4){\cong_{\UCS}}\\
\left\langle T,g\left|\hat{\rho}_{1}\left(R_{1}\right)\cup M_{1}\cup\hat{\rho}_{2}\left(R_{2}\right)\cup M_{2}\right.\right\rangle_{\UCS}
}\]
Fix $s\in S_{1}$.  In $\left\langle T,g\left|\hat{\rho}_{1}\left(R_{1}\right)\cup M_{1}\cup\hat{\rho}_{2}\left(R_{2}\right)\cup M_{2}\right.\right\rangle_{\UCS}$, $\left(\sigma\circ\alpha\circ q_{1}\right)(s)$ is the generator $[s]$, and $\left(\tau\circ\beta\circ\phi\circ q_{1}\right)(s)=\left(\tau\circ\beta\circ q_{2}\right)\left(\psi_{1}(s)\right)$ is $\left[\hat{\rho}_{2}\left(\psi_{1}(s)\right)\right]$.  Also,
\[
[s]
=\left[s-\left(\hat{\rho}_{2}\circ\Theta_{2}\right)(s)\right]+\left[\left(\hat{\rho}_{2}\circ\Theta_{2}\right)(s)\right]
=\left[\left(\hat{\rho}_{2}\circ\Theta_{2}\right)(s)\right]
=\left[\hat{\rho}_{2}\left(\psi_{1}(s)\right)\right]
\]
in $\left\langle T,g\left|\hat{\rho}_{1}\left(R_{1}\right)\cup M_{1}\cup\hat{\rho}_{2}\left(R_{2}\right)\cup M_{2}\right.\right\rangle_{\UCS}$.  Thus,
\[
\left(\tau\circ\beta\circ\phi\circ q_{1}\right)(s)=\left(\sigma\circ\alpha\circ q_{1}\right)(s).
\]
As $s\in S_{1}$ was arbitrary, Theorem \ref{univprop} states that $\tau\circ\beta\circ\phi\circ q_{1}=\sigma\circ\alpha\circ q_{1}$.  By the universal property of the quotient, $\tau\circ\beta\circ\phi=\sigma\circ\alpha$.  As $\tau$ and $\beta$ are invertible, $\phi=\beta^{-1}\circ\tau^{-1}\circ\sigma\circ\alpha$.  Hence, $\phi$ is a sequence of isomorphisms given by Tietze transformations.

\end{proof}

Now, a Tietze transformation is \emph{elementary} if only one generator or C*-relation is changed.  As such, any Tietze transformation where finitely many changes are made can be realized by a finite sequence of elementary Tietze transformations.  The terms \emph{finitely generated}, \emph{finitely related}, and \emph{finitely presented} already exist in the literature and coincide readily with a finite generation set, finite relation set, and finite generation and relation sets in this context.  Appealing to \cite[Proposition 41]{hadwin2003}, every finitely generated unital C*-algebra is finitely presented.  Thus, the following corollary is the direct analog of the result from \cite{tietze}.

\begin{cor}\label{tietze2}
Given unital C*-algebras $\alg{A}_{1}$ and $\alg{A}_{2}$ are finitely generated in $\UCS$, $\alg{A}_{1}\cong_{\UCS}\alg{A}_{2}$ if and only if there is a finite sequence of elementary Tietze transformations changing the presentation of $\alg{A}_{1}$ into the presentation for $\alg{A}_{2}$.
\end{cor}

\section{Examples}\label{examples}

With the main results proven, this section holds some examples of applying Tietze transformations to presentations of C*-algebras.  Several examples are worked in \cite[\S2.4.2,3.5-6]{gerbracht} for *-algebraic relations.  The examples in the current work will be distinct and make more use of C*-relations from the functional calculus.

The first two examples are elementary, but pedagogically useful in demonstrating the implications of Theorem \ref{tietze1}.  They also demonstrate a restriction imposed when considering the general category $\CS$, rather than the unital category $\UCS$.

\begin{ex}[Positivity and Self-Adjointness, $\UCS$]\label{example1}
Recall that
\[
\left\langle(x,1)\left|x=x^{*}\right.\right\rangle_{\UCS}\cong_{\UCS}C[0,1]
\]
and
\[
\langle(y,1)|y\geq0\rangle_{\UCS}\cong_{\UCS}C[0,1],
\]
so Corollary \ref{tietze2} states there is a sequence of elementary Tietze transformations changing the first presentation into the second.  Letting $y=\frac{1}{2}x+\frac{1}{2}\mathbb{1}$, one such sequence would be the following.
\[\begin{array}{rcl}
\left\langle (x,1)\left|x=x^{*}\right.\right\rangle_{\UCS}	&	\cong_{\UCS}	&	\left\langle (x,1),(y,1)\left|\begin{array}{c}x=x^{*},\\ y=\frac{1}{2}x+\frac{1}{2}\mathbb{1}\end{array}\right.\right\rangle_{\UCS}\\[20pt]
&	\cong_{\UCS}	&	\left\langle (x,1),(y,1)\left|\begin{array}{c}x=x^{*},\\ y=\frac{1}{2}x+\frac{1}{2}\mathbb{1},y\geq 0\end{array}\right.\right\rangle_{\UCS}\\[20pt]
&	\cong_{\UCS}	&	\left\langle (x,1),(y,1)\left|\begin{array}{c}x=x^{*},\\ y=\frac{1}{2}x+\frac{1}{2}\mathbb{1},y\geq 0\\ x=2y-\mathbb{1}\end{array}\right.\right\rangle_{\UCS}\\[25pt]
&	\cong_{\UCS}	&	\left\langle (x,1),(y,1)\left|\begin{array}{c}x=x^{*},\\ y\geq 0\\ x=2y-\mathbb{1}\end{array}\right.\right\rangle_{\UCS}\\[20pt]
&	\cong_{\UCS}	&	\left\langle (x,1),(y,1)\left|\begin{array}{c}y\geq 0\\ x=2y-\mathbb{1}\end{array}\right.\right\rangle_{\UCS}\\[15pt]
&	\cong_{\UCS}	&	\left\langle (y,1)\left|y\geq 0\right.\right\rangle_{\UCS}\\
\end{array}\]
\end{ex}

\begin{ex}[Positivity and Self-Adjointness, $\CS$]
Recall that
\[
\left\langle(x,1)\left|x=x^{*}\right.\right\rangle_{\CS}\cong_{\CS}C_{0}\left([-1,0)\cup(0,1]\right)
\]
and
\[
\langle(y,1)|y\geq0\rangle_{\CS}\cong_{\CS}C_{0}(0,1].
\]
The analogous theorem to Theorem \ref{tietze1} for $\CS$ does hold by a nearly identical proof.  Thus, no sequence of Tietze transformations can exist to change the first presentation into the second.  This shows that any sequence of transformations for Example \ref{example1} must involve the unit.
\end{ex}

\begin{ex}[Left-Invertibility]
For $\lambda,\mu\geq0$, consider the unital C*-algebra of a left-invertible element,
\[
\alg{L}:=\left\langle(x,\lambda)\left|\mu^{2}x^{*}x\geq\mathbb{1}\right.\right\rangle_{\UCS}.
\]
If $\lambda\mu<1$, $\|\mathbb{1}\|_{\alg{L}}<1$, meaning $\mathbb{1}=0$ in $\alg{L}$.  Hence, $\alg{L}\cong_{\UCS}\mathbb{O}$, the zero algebra.

For $\lambda\mu\geq1$, the generator $x$ can be split via the continuous functional calculus and the polar decomposition into
\[
q:=\left(x^{*}x\right)^{\frac{1}{2}}
\]
and
\[
u:=\mu x\left(p\left(\mu\left(x^{*}x\right)^{\frac{1}{2}}-\mathbb{1}\right)+\mathbb{1}\right)^{-1},
\]
where $p:\mathbb{R}\to\mathbb{R}$ is defined as in Example \ref{relations}.  Then, Tietze transformations and use of the free product characterize the algebra.  Verification of each C*-relation's validity is an exercise in the continuous functional calculus.  Here, $\alg{T}$ is the Toeplitz algebra.
\[\begin{array}{rcl}
\alg{L}	&	\cong_{\UCS}	&	\left\langle(x,\lambda),(q,\lambda)\left|\mu^{2}x^{*}x\geq\mathbb{1}, q=\left(x^{*}x\right)^{\frac{1}{2}}\right.\right\rangle_{\UCS}\\[15pt]

&	\cong_{\UCS}	&	\left\langle\begin{array}{c}(x,\lambda),(q,\lambda),\\ (u,\lambda\mu)\end{array}\left|\begin{array}{c}\mu^{2}x^{*}x\geq\mathbb{1},q=\left(x^{*}x\right)^{\frac{1}{2}},\\ u=\mu x\left(p\left(\mu\left(x^{*}x\right)^{\frac{1}{2}}-\mathbb{1}\right)+\mathbb{1}\right)^{-1}\end{array}\right.\right\rangle_{\UCS}\\[20pt]

&	\cong_{\UCS}	&	\left\langle\begin{array}{c}(x,\lambda),(q,\lambda),\\ (u,\lambda\mu)\end{array}\left|\begin{array}{c}\mu^{2}x^{*}x\geq\mathbb{1},q=\left(x^{*}x\right)^{\frac{1}{2}},\\ u=\mu x\left(p\left(\mu\left(x^{*}x\right)^{\frac{1}{2}}-\mathbb{1}\right)+\mathbb{1}\right)^{-1},\\ \mathbb{1}\leq\mu q\end{array}\right.\right\rangle_{\UCS}\\[20pt]

&	\cong_{\UCS}	&	\left\langle\begin{array}{c}(x,\lambda),(q,\lambda),\\ (u,\lambda\mu)\end{array}\left|\begin{array}{c}\mu^{2}x^{*}x\geq\mathbb{1},q=\left(x^{*}x\right)^{\frac{1}{2}},\\ u=\mu x\left(p\left(\mu\left(x^{*}x\right)^{\frac{1}{2}}-\mathbb{1}\right)+\mathbb{1}\right)^{-1},\\ \mathbb{1}\leq\mu q, u^{*}u=\mathbb{1}\end{array}\right.\right\rangle_{\UCS}\\[20pt]

&	\cong_{\UCS}	&	\left\langle\begin{array}{c}(x,\lambda),(q,\lambda),\\ (u,\lambda\mu)\end{array}\left|\begin{array}{c}\mu^{2}x^{*}x\geq\mathbb{1},q=\left(x^{*}x\right)^{\frac{1}{2}},\\ u=\mu x\left(p\left(\mu\left(x^{*}x\right)^{\frac{1}{2}}-\mathbb{1}\right)+\mathbb{1}\right)^{-1},\\ \mathbb{1}\leq\mu q, u^{*}u=\mathbb{1},x=uq\end{array}\right.\right\rangle_{\UCS}\\[20pt]

&	\cong_{\UCS}	&	\left\langle\begin{array}{c}(x,\lambda),(q,\lambda),\\ (u,\lambda\mu)\end{array}\left|\begin{array}{c}q=\left(x^{*}x\right)^{\frac{1}{2}},\\ u=\mu x\left(p\left(\mu\left(x^{*}x\right)^{\frac{1}{2}}-\mathbb{1}\right)+\mathbb{1}\right)^{-1},\\ \mathbb{1}\leq\mu q, u^{*}u=\mathbb{1},x=uq\end{array}\right.\right\rangle_{\UCS}\\[20pt]

&	\cong_{\UCS}	&	\left\langle\begin{array}{c}(x,\lambda),(q,\lambda),\\ (u,\lambda\mu)\end{array}\left|\begin{array}{c}u=\mu x\left(p\left(\mu\left(x^{*}x\right)^{\frac{1}{2}}-\mathbb{1}\right)+\mathbb{1}\right)^{-1},\\ \mathbb{1}\leq\mu q, u^{*}u=\mathbb{1},x=uq\end{array}\right.\right\rangle_{\UCS}\\[20pt]

&	\cong_{\UCS}	&	\left\langle\left.\begin{array}{c}(x,\lambda),(q,\lambda),\\ (u,\lambda\mu)\end{array}\right|\mathbb{1}\leq\mu q, u^{*}u=\mathbb{1},x=uq\right\rangle_{\UCS}\\[20pt]

&	\cong_{\UCS}	&	\left\langle(q,\lambda),(u,\lambda\mu)\left|\mathbb{1}\leq\mu q, u^{*}u=\mathbb{1}\right.\right\rangle_{\UCS}\\[15pt]

&	\cong_{\UCS}	&	\langle(q,\lambda)|\mathbb{1}\leq\mu q\rangle_{\UCS}*_{\mathbb{C}}\left\langle(u,\lambda\mu)\left|u^{*}u=\mathbb{1}\right.\right\rangle_{\UCS}\\[15pt]

&	\cong_{\UCS}	&	C\left[\frac{1}{\mu},\lambda\right]*_{\mathbb{C}}\alg{T}\\[15pt]

\end{array}\]
In summary,
\[\begin{array}{rcl}
\left\langle(x,\lambda)\left|\mu^{2}x^{*}x\geq\mathbb{1}\right.\right\rangle_{\UCS}	&	\cong_{\UCS}	&	\left\{\begin{array}{cc}
\mathbb{O},	&	\lambda\mu<1,\\
\alg{T},	&	\lambda\mu=1,\\
C[0,1]*_{\mathbb{C}}\alg{T},	&	\lambda\mu>1.\\
\end{array}\right.\\
\end{array}\]
Similar transformations can be used to prove the following isomorphisms for the unital C*-algebras of a single right-invertible or a single true invertible.
\[\begin{array}{rcl}
\left\langle(x,\lambda)\left|\mu^{2}xx^{*}\geq\mathbb{1}\right.\right\rangle_{\UCS}	&	\cong_{\UCS}	&	\left\{\begin{array}{cc}
\mathbb{O},	&	\lambda\mu<1,\\
\alg{T},	&	\lambda\mu=1,\\
C[0,1]*_{\mathbb{C}}\alg{T},	&	\lambda\mu>1.\\
\end{array}\right.\\[20pt]
\left\langle(x,\lambda)\left|\mu^{2}x^{*}x\geq\mathbb{1},\mu^{2}xx^{*}\geq\mathbb{1}\right.\right\rangle_{\UCS}	&	\cong_{\UCS}	&	\left\{\begin{array}{cc}
\mathbb{O},	&	\lambda\mu<1,\\
C(\mathbb{T}),	&	\lambda\mu=1,\\
C[0,1]*_{\mathbb{C}}C(\mathbb{T}),	&	\lambda\mu>1.\\
\end{array}\right.\\
\end{array}\]
\end{ex}

\begin{ex}[Idempotency]
For $\lambda\geq0$, consider the unital C*-algebra of a single idempotent element,
\[
\alg{A}:=\left\langle(x,\lambda)\left|x=x^{2}\right.\right\rangle_{\UCS}.
\]
If $\lambda<1$, then $x=0$.  Hence, $\alg{A}\cong_{\UCS}\mathbb{C}$.

For $\lambda\geq1$, the formula for the range projection of an idempotent from \cite[Proposition IV.1.1]{davidson} will be used on $x$ and $x^{*}$, giving two new generators:
\[
r:=xx^{*}\left(\mathbb{1}+\left(x-x^{*}\right)^{*}\left(x-x^{*}\right)\right)^{-1}
\]
and
\[
k:=(\mathbb{1}-x)(\mathbb{1}-x)^{*}\left(\mathbb{1}+\left(x^{*}-x\right)^{*}\left(x^{*}-x\right)\right)^{-1}.
\]
To use \cite[Theorem 1]{vidav}, define $f_{\lambda}:[0,1]\to\mathbb{C}$ by
\[
f_{\lambda}(\nu):=\left\{\begin{array}{cc}
\nu,	&	0\leq\nu\leq\sqrt{1-\lambda^{-2}},\\
\frac{\sqrt{1-\lambda^{-2}}}{\sqrt{1-\lambda^{-2}}-1}(\nu-1),	&	\sqrt{1-\lambda^{-2}}<\mu\leq 1.
\end{array}\right.
\]
This function ensures that the inverse in the formula for $x$ in terms of $r$ and $k$ exists in $\langle (x,\lambda),(r,1),(k,1)|\emptyset\rangle_{\UCS}$ before quotienting:
\[
x=\left(\mathbb{1}-f_{\lambda}\left(rk^{*}kr^{*}\right)\right)^{-1}(r-rk).
\]
Thus, Tietze transformations give an alternative realization of this algebra.  Verification of each C*-relation's validity is an exercise in the continuous functional calculus.
\[\begin{array}{rcl}
\alg{A}	&	\cong_{\UCS}	&	\left\langle(x,\lambda),(r,1)\left|\begin{array}{c}x=x^{2},\\ r=xx^{*}\left(\mathbb{1}+\left(x-x^{*}\right)^{*}\left(x-x^{*}\right)\right)^{-1}\end{array}\right.\right\rangle_{\UCS}\\[10pt]

&	\cong_{\UCS}	&	\left\langle\begin{array}{c}(x,\lambda),\\ (r,1),(k,1)\end{array}\left|\begin{array}{c}x=x^{2},\\ r=xx^{*}\left(\mathbb{1}+\left(x-x^{*}\right)^{*}\left(x-x^{*}\right)\right)^{-1},\\ k=(\mathbb{1}-x)(\mathbb{1}-x)^{*}\left(\mathbb{1}+\left(x^{*}-x\right)^{*}\left(x^{*}-x\right)\right)^{-1}\end{array}\right.\right\rangle_{\UCS},\\[20pt]

&	\cong_{\UCS}	&	\left\langle\begin{array}{c}(x,\lambda),\\ (r,1),(k,1)\end{array}\left|\begin{array}{c}x=x^{2},r^{2}=r\\ r=xx^{*}\left(\mathbb{1}+\left(x-x^{*}\right)^{*}\left(x-x^{*}\right)\right)^{-1},\\ k=(\mathbb{1}-x)(\mathbb{1}-x)^{*}\left(\mathbb{1}+\left(x^{*}-x\right)^{*}\left(x^{*}-x\right)\right)^{-1}\end{array}\right.\right\rangle_{\UCS}\\[20pt]

&	\cong_{\UCS}	&	\left\langle\begin{array}{c}(x,\lambda),\\ (r,1),(k,1)\end{array}\left|\begin{array}{c}x=x^{2},r^{2}=r^{*}=r\\ r=xx^{*}\left(\mathbb{1}+\left(x-x^{*}\right)^{*}\left(x-x^{*}\right)\right)^{-1},\\ k=(\mathbb{1}-x)(\mathbb{1}-x)^{*}\left(\mathbb{1}+\left(x^{*}-x\right)^{*}\left(x^{*}-x\right)\right)^{-1}\end{array}\right.\right\rangle_{\UCS}\\[20pt]

&	\cong_{\UCS}	&	\left\langle\begin{array}{c}(x,\lambda),\\ (r,1),(k,1)\end{array}\left|\begin{array}{c}x=x^{2},r^{2}=r^{*}=r,k^{2}=k\\ r=xx^{*}\left(\mathbb{1}+\left(x-x^{*}\right)^{*}\left(x-x^{*}\right)\right)^{-1},\\ k=(\mathbb{1}-x)(\mathbb{1}-x)^{*}\left(\mathbb{1}+\left(x^{*}-x\right)^{*}\left(x^{*}-x\right)\right)^{-1}\end{array}\right.\right\rangle_{\UCS}\\[20pt]

&	\cong_{\UCS}	&	\left\langle\begin{array}{c}(x,\lambda),\\ (r,1),(k,1)\end{array}\left|\begin{array}{c}x=x^{2},r^{2}=r^{*}=r,k^{2}=k^{*}=k\\ r=xx^{*}\left(\mathbb{1}+\left(x-x^{*}\right)^{*}\left(x-x^{*}\right)\right)^{-1},\\ k=(\mathbb{1}-x)(\mathbb{1}-x)^{*}\left(\mathbb{1}+\left(x^{*}-x\right)^{*}\left(x^{*}-x\right)\right)^{-1}\end{array}\right.\right\rangle_{\UCS}\\[20pt]

&	\cong_{\UCS}	&	\left\langle\begin{array}{c}(x,\lambda),\\ (r,1),(k,1)\end{array}\left|\begin{array}{c}x=x^{2},r^{2}=r^{*}=r,k^{2}=k^{*}=k,\|rk\|\leq\sqrt{1-\lambda^{-2}}\\ r=xx^{*}\left(\mathbb{1}+\left(x-x^{*}\right)^{*}\left(x-x^{*}\right)\right)^{-1},\\ k=(\mathbb{1}-x)(\mathbb{1}-x)^{*}\left(\mathbb{1}+\left(x^{*}-x\right)^{*}\left(x^{*}-x\right)\right)^{-1}\end{array}\right.\right\rangle_{\UCS}\\[25pt]

&	\cong_{\UCS}	&	\left\langle\begin{array}{c}(x,\lambda),\\ (r,1),(k,1)\end{array}\left|\begin{array}{c}x=x^{2},r^{2}=r^{*}=r,k^{2}=k^{*}=k,\|rk\|\leq\sqrt{1-\lambda^{-2}}\\ r=xx^{*}\left(\mathbb{1}+\left(x-x^{*}\right)^{*}\left(x-x^{*}\right)\right)^{-1},\\ k=(\mathbb{1}-x)(\mathbb{1}-x)^{*}\left(\mathbb{1}+\left(x^{*}-x\right)^{*}\left(x^{*}-x\right)\right)^{-1},\\ x=\left(\mathbb{1}-f_{\lambda}\left(rk^{*}kr^{*}\right)\right)^{-1}(r-rk)\end{array}\right.\right\rangle_{\UCS}\\[30pt]

&	\cong_{\UCS}	&	\left\langle\begin{array}{c}(x,\lambda),\\ (r,1),(k,1)\end{array}\left|\begin{array}{c}x=x^{2},r^{2}=r^{*}=r,k^{2}=k^{*}=k,\|rk\|\leq\sqrt{1-\lambda^{-2}}\\  k=(\mathbb{1}-x)(\mathbb{1}-x)^{*}\left(\mathbb{1}+\left(x^{*}-x\right)^{*}\left(x^{*}-x\right)\right)^{-1},\\ x=\left(\mathbb{1}-f_{\lambda}\left(rk^{*}kr^{*}\right)\right)^{-1}(r-rk)\end{array}\right.\right\rangle_{\UCS}\\[25pt]

&	\cong_{\UCS}	&	\left\langle\begin{array}{c}(x,\lambda),\\ (r,1),(k,1)\end{array}\left|\begin{array}{c}x=x^{2},r^{2}=r^{*}=r,k^{2}=k^{*}=k,\|rk\|\leq\sqrt{1-\lambda^{-2}}\\ x=\left(\mathbb{1}-f_{\lambda}\left(rk^{*}kr^{*}\right)\right)^{-1}(r-rk)\end{array}\right.\right\rangle_{\UCS}\\[15pt]

&	\cong_{\UCS}	&	\left\langle\begin{array}{c}(x,\lambda),\\ (r,1),(k,1)\end{array}\left|\begin{array}{c}r^{2}=r^{*}=r,k^{2}=k^{*}=k,\|rk\|\leq\sqrt{1-\lambda^{-2}}\\ x=\left(\mathbb{1}-f_{\lambda}\left(rk^{*}kr^{*}\right)\right)^{-1}(r-rk)\end{array}\right.\right\rangle_{\UCS}\\[15pt]

&	\cong_{\UCS}	&	\left\langle(r,1),(k,1)\left|r^{2}=r^{*}=r,k^{2}=k^{*}=k,\|rk\|\leq\sqrt{1-\lambda^{-2}}\right.\right\rangle_{\UCS}\\
\end{array}\]
Now that $\alg{A}$ has been written as a C*-algebra of two projections, the result of \cite[Theorem 3.2]{pedersen2} gives
\[\begin{array}{rcl}
\alg{A}	&	\cong_{\UCS}	&	\begin{bmatrix}C(X)	&	C_{0}\left(X\setminus\{0,1\}\right)\\ C_{0}\left(X\setminus\{0,1\}\right)	&	C(X)\end{bmatrix},
\end{array}\]
where $X:=\sigma_{\alg{A}}(rkr)$.  Use of the universal property shows that
\[
X=\left[0,1-\lambda^{-2}\right].
\]
In summary,
\[\begin{array}{rcl}
\alg{A}	&	\cong_{\UCS}	&	\left\{\begin{array}{cc}
\mathbb{C},	&	\lambda<1,\\[8pt]
\mathbb{C}^{2},	&	\lambda=1,\\[8pt]
\begin{bmatrix}C[0,1]	&	C_{0}(0,1]\\[8pt] C_{0}(0,1]	&	C[0,1]\end{bmatrix},	&	\lambda>1.\\
\end{array}\right.\\
\end{array}\]
For the non-unital case, Theorem \ref{unit-present} gives the following characterization by recognizing the ideal generated by $x$ in $\alg{A}$.
\[\begin{array}{rcl}
\left\langle(x,\lambda)\left|x=x^{2}\right.\right\rangle_{\CS}	&	\cong_{\CS}	&	\left\{\begin{array}{cc}
\mathbb{O},	&	\lambda<1,\\[8pt]
\mathbb{C},	&	\lambda=1,\\[8pt]
\begin{bmatrix}C[0,1]	&	C_{0}(0,1]\\[8pt] C_{0}(0,1]	&	C_{0}(0,1]\end{bmatrix},	&	\lambda>1.\\
\end{array}\right.\\
\end{array}\]
\end{ex}


\begin{bibdiv}
\begin{biblist}

\bib{ara2011}{article}{
      author={Ara, Pere},
      author={Corti{\~n}as, Guillermo},
       title={Tensor products of leavitt path algebras},
        date={201108},
      eprint={1108.0352v1},
         url={http://arxiv.org/abs/1108.0352v1},
}

\bib{baumslag}{book}{
      author={Baumslag, Gilbert},
       title={Topics in combinatorial group theory},
      series={Lectures in Mathematics ETH Z{\"u}rich},
   publisher={Birkh{\"a}user Verlag},
     address={Basel},
        date={1993},
        ISBN={3-7643-2921-1},
      review={\MR{1243634 (94j:20034)}},
}

\bib{blackadar1985}{article}{
      author={Blackadar, Bruce},
       title={Shape theory for {$C^\ast$}-algebras},
        date={1985},
        ISSN={0025-5521},
     journal={Math. Scand.},
      volume={56},
      number={2},
       pages={249\ndash 275},
      review={\MR{MR813640 (87b:46074)}},
}

\bib{davidson}{book}{
      author={Davidson, Kenneth~R.},
       title={{$C^*$}-algebras by example},
      series={Fields Institute Monographs},
   publisher={American Mathematical Society},
     address={Providence, RI},
        date={1996},
      volume={6},
        ISBN={0-8218-0599-1},
      review={\MR{MR1402012 (97i:46095)}},
}

\bib{gerbracht}{thesis}{
      author={Gerbracht, Eberhard H.-A.},
       title={Elemente einer kombinatorischen theorie der c*-algebren:
  Pr{\"a}sentationen von c*-algebren mittels erzeugender und relationen},
        type={Ph.D. Thesis},
        date={1998},
}

\bib{goodearl}{article}{
      author={Goodearl, K.~R.},
      author={Menal, P.},
       title={Free and residually finite-dimensional {$C^*$}-algebras},
        date={1990},
        ISSN={0022-1236},
     journal={J. Funct. Anal.},
      volume={90},
      number={2},
       pages={391\ndash 410},
         url={http://dx.doi.org/10.1016/0022-1236(90)90089-4},
      review={\MR{MR1052340 (91f:46078)}},
}

\bib{grilliette1}{article}{
      author={Grilliette, Will},
       title={Scaled-free objects},
        date={201011},
      eprint={arXiv:1011.0717v2},
         url={http://arxiv.org/abs/1011.0717v2},
}

\bib{hadwin2003}{article}{
      author={Hadwin, Don},
      author={Kaonga, Llolsten},
      author={Mathes, Ben},
       title={Noncommutative continuous functions},
        date={2003},
        ISSN={0304-9914},
     journal={J. Korean Math. Soc.},
      volume={40},
      number={5},
       pages={789\ndash 830},
      review={\MR{MR1996841 (2005g:46132)}},
}

\bib{loring1}{book}{
      author={Loring, Terry~A.},
       title={Lifting solutions to perturbing problems in {$C^*$}-algebras},
      series={Fields Institute Monographs},
   publisher={American Mathematical Society},
     address={Providence, RI},
        date={1997},
      volume={8},
        ISBN={0-8218-0602-5},
      review={\MR{MR1420863 (98a:46090)}},
}

\bib{loring2008}{article}{
      author={Loring, Terry~A.},
       title={{$C^*$}-algebra relations},
        date={2010},
        ISSN={0025-5521},
     journal={Math. Scand.},
      volume={107},
      number={1},
       pages={43\ndash 72},
      review={\MR{2679392}},
}

\bib{pedersen2}{article}{
      author={Pedersen, Gert~Kjaerg{\.a}rd},
       title={Measure theory for {$C^{\ast} $} algebras. {II}},
        date={1968},
        ISSN={0025-5521},
     journal={Math. Scand.},
      volume={22},
       pages={63\ndash 74},
      review={\MR{MR0246138 (39 \#7444)}},
}

\bib{rordam1994}{article}{
      author={R{\o}rdam, Mikael},
       title={A short proof of {E}lliott's theorem:
  {$\mathcal{O}_2\otimes\mathcal{O}_2\cong\mathcal{O}_2$}},
        date={1994},
        ISSN={0706-1994},
     journal={C. R. Math. Rep. Acad. Sci. Canada},
      volume={16},
      number={1},
       pages={31\ndash 36},
      review={\MR{1276341 (95d:46064)}},
}

\bib{tietze}{article}{
      author={Tietze, Heinrich},
       title={\"{U}ber die topologischen {I}nvarianten mehrdimensionaler
  {M}annigfaltigkeiten},
        date={1908},
        ISSN={0026-9255},
     journal={Monatsh. Math. Phys.},
      volume={19},
      number={1},
       pages={1\ndash 118},
         url={http://dx.doi.org/10.1007/BF01736688},
      review={\MR{MR1547755}},
}

\bib{vidav}{article}{
      author={Vidav, Ivan},
       title={On idempotent operators in a {H}ilbert space},
        date={1964},
        ISSN={0350-1302},
     journal={Publ. Inst. Math. (Beograd) (N.S.)},
      volume={4 (18)},
       pages={157\ndash 163},
      review={\MR{MR0171161 (30 \#1392)}},
}

\bib{freerandom}{book}{
      author={Voiculescu, D.~V.},
      author={Dykema, K.~J.},
      author={Nica, A.},
       title={Free random variables},
      series={CRM Monograph Series},
   publisher={American Mathematical Society},
     address={Providence, RI},
        date={1992},
      volume={1},
        ISBN={0-8218-6999-X},
        note={A noncommutative probability approach to free products with
  applications to random matrices, operator algebras and harmonic analysis on
  free groups},
      review={\MR{MR1217253 (94c:46133)}},
}

\end{biblist}
\end{bibdiv}
\end{document}